\documentclass[12pt]{extarticle}

\usepackage[margin=1.5in]{geometry}  % set the margins to 1in on all sides
\usepackage{graphicx}              % to include figures
\usepackage{amsmath}               % great math stuff
\usepackage{amsfonts}              % for blackboard bold, etc
\usepackage{amsthm}                % better theorem environments
\usepackage{framed}
\usepackage{comment}
\usepackage{url}
\newtheorem{thm}{Theorem}[section]
\newtheorem{lem}[thm]{Lemma}

\begin{document}

\nocite{*}

\title{\bf On Solving a Curious Inequality \\ of Ramanujan}

\author{\textsc{Adrian W. Dudek} \\ 
Mathematical Sciences Institute \\
The Australian National University \\ 
\texttt{adrian.dudek@anu.edu.au}\\
 \\
\textsc{David J. Platt} \\ 
 Heilbronn Institute for Mathematical Research\\
University of Bristol, Bristol, UK \\ 
\texttt{dave.platt@bris.ac.uk}
}
\date{}

\maketitle

\begin{abstract}
\noindent Ramanujan proved that the inequality 

$$\pi(x)^2 < \frac{e x}{\log x} \pi\Big(\frac{x}{e}\Big)$$
holds for all sufficiently large values of $x$. Using an explicit estimate for the error in the prime number theorem, we show unconditionally that it holds if $x \geq \exp(9658)$. Furthermore, we solve the inequality completely on the Riemann Hypothesis, and show that $x=38, 358, 837, 682$ is the largest integer counterexample.

\end{abstract}

\section{Introduction}

We let $\pi(x)$ denote the number of primes which are less than or equal to $x$. In one of his notebooks, Ramanujan (see the preservations by Berndt \cite[Ch.24]{Berndt}) proved that the inequality

\begin{equation} \label{inequality}
\pi(x)^2 < \frac{e x}{\log x} \pi\Big(\frac{x}{e}\Big)
\end{equation}
holds for all sufficiently large values of $x$. Berndt \cite{Berndt} states that Wheeler, Keiper and Galway used \textsc{Mathematica} in an attempt to determine an $x_0$ such that (\ref{inequality}) holds for all $x \geq x_0$. They were unsuccessful, but independently Galway was able to establish that the largest prime counter-example below $10^{11}$ occurs at $x = 38, 358, 837, 677$. 
\begin{comment}
Indeed, one can run a straightforward check to see that $x = 38, 358, 837, 682$ is the greatest integer counterexample for $x<10^{11}$.
\end{comment}

Hassani looked at the problem in $2012$ \cite{Hassani} and established (inter alia):
\begin{thm}[Hassani]\label{thm:hassani}
If one assumes the Riemann Hypothesis, then inequality (\ref{inequality}) holds for all $x\geq 138,766,146,692,471,228$.
\end{thm}
\begin{proof}
This is Theorem 1.2 of \cite{Hassani}.
\end{proof}

\begin{comment}
The first purpose of this paper is to show that Ramanujan's inequality holds without condition for all $x \geq \exp(9658)$. We then solve the inequality completely on the assumption of the Riemann hypothesis and show that $x=38, 358, 837, 682$ is the greatest integer counterexample, and so (\ref{inequality}) holds conditionally for all real numbers $x \geq 38, 358, 837, 683$. 
\end{comment}
The purpose of this paper is to establish the following two theorems:
\begin{thm}\label{thm:uncon}
Inequality (\ref{inequality}) holds unconditionally for all $x\geq\exp(9658)$.
\end{thm}

\begin{thm}\label{thm:con}
Assuming the Riemann Hypothesis, the largest integer counter-example to inequality (\ref{inequality}) is that at $x=38, 358, 837, 682$.
\end{thm}

We will look at the unconditional result first, before considering that contingent on the Riemann Hypothesis.

\section{The unconditional result}

\subsection{Ramanujan's original proof}

We start by giving Ramanujan's original and, we think, rather fetching proof, which is based on the prime number theorem, or more specifically that

\begin{equation} \label{piexpansion}
\pi(x) = x \sum_{k=0}^{4} \frac{k!}{\log^{k+1}x}+O\Big(\frac{x}{\log^6 x}\Big)
\end{equation}
as $x \rightarrow \infty$. As such we have the two estimates

\begin{equation*} \label{pisquared}
\pi^2 (x) = x^2 \Big\{ \frac{1}{\log^2 x} + \frac{2}{\log^3 x} + \frac{5}{\log^4 x} + \frac{16}{\log^5 x} + \frac{64}{\log^6 x} \Big\} + O\Big( \frac{x^2}{\log^7 x}\Big)
\end{equation*}
and

\begin{eqnarray*} \label{pimult}
\frac{e x}{\log x} \pi \Big(\frac{x}{e}\Big) & = & \frac{x^2}{\log x} \Big\{ \sum_{k=0}^{4} \frac{k!}{(\log x-1)^{k+1}} \Big\} + O\Big( \frac{x^2}{\log^7 x} \Big)  \\
& = & x^2 \Big\{ \frac{1}{\log^2 x} + \frac{2}{\log^3 x} + \frac{5}{\log^4 x} + \frac{16}{\log^5 x} + \frac{65}{\log^6 x} \Big\} + O\Big( \frac{x^2}{\log^7 x}\Big).
\end{eqnarray*}
Subtracting the above two expressions gives

\begin{equation} \label{makemenegative}
\pi^2 (x) - \frac{e x}{\log x} \pi \Big(\frac{x}{e}\Big) = - \frac{x^2}{\log^6 x} + O \Big(\frac{x^2}{\log^7 x} \Big)
\end{equation}
which is negative for sufficiently large values of $x$. This completes the proof.

The proof serves as a tribute to the workings of Ramanujan's mind, for surely one would not calculate the asymptotic expansions of such functions without the knowledge that doing so would be fruitful. 

Note that if one were to work through the above proof using explicit estimates on the asymptotic expansion of the prime-counting function, then one would be able to make precise what is meant by ``sufficiently large''. The following lemma shows how one might do this.

\begin{lem} \label{lem1}
Let $m_a, M_a \in \mathbb{R}$  and suppose that for $x>x_a$ we have

$$ x \sum_{k=0}^{4} \frac{k!}{\log^{k+1}x}+ \frac{m_a x}{\log^6 x} < \pi(x) < x \sum_{k=0}^{4} \frac{k!}{\log^{k+1}x}+\frac{M_a x}{\log^6 x}.$$
Then Ramanujan's inequality is true if 

$$x > \max( e x_{a},x_{a}' )$$
where a value for $x_{a}'$ can be obtained in the proof and is completely determined by $m_a, M_a$ and $x_{a}$.

\end{lem}

\begin{proof}
Following along the lines of Ramanujan's proof we have for $x > x_{a}$

\begin{equation} \label{pipi}
\pi^2(x)  <  x^2 \Big\{ \frac{1}{\log^2 x}+ \frac{2}{\log^3 x}+ \frac{5}{\log^4 x}+ \frac{16}{\log^5 x}+ \frac{64}{\log^6 x} + \frac{\epsilon_{M_a}(x)}{\log^7 x} \Big\} 
\end{equation}
where 

$$\epsilon_{M_a} (x) = 72 + 2 M_a + \frac{2M_a+132}{\log x} + \frac{4M_a+288}{\log^2 x} + \frac{12 M_a+576}{\log^3 x}+\frac{48M_a}{\log^4 x} + \frac{M_a^2}{\log^5 x}.$$

The other term requires slightly more trickery; we have for $x > e x_{a}$

$$\frac{ex}{\log x} \pi \Big(\frac{x}{e} \Big) > \frac{x^2}{\log x} \Big( \sum_{k=0}^{4} \frac{k!}{(\log x - 1)^{k+1}}\Big) + \frac{m_a x}{(\log x-1)^{6}}. $$
We make use of the inequality

\begin{eqnarray*} 
\frac{1}{(\log x - 1)^{k+1}} & = & \frac{1}{\log^{k+1} x} \Big(1+ \frac{1}{\log x} + \frac{1}{\log^2 x} + \frac{1}{\log^3 x} + \cdots \Big)^{k+1} \\ \\
& > & \frac{1}{\log^{k+1} x} \Big(1+ \frac{1}{\log x}+ \cdots + \frac{1}{\log^{5-k} x} \Big)^{k+1}
\end{eqnarray*}
to get

\begin{equation} \label{epi}
\frac{ex}{\log x} \pi \Big(\frac{x}{e} \Big) > x^2 \Big\{ \frac{1}{\log^2 x}+ \frac{2}{\log^3 x}+ \frac{5}{\log^4 x}+ \frac{16}{\log^5 x}+ \frac{64}{\log^6 x} + \frac{\epsilon_{m_a}(x)}{\log^7 x} \Big\}, 
\end{equation}
where

$$\epsilon_{m_a}(x) = 206+m_a+\frac{364}{\log x} + \frac{381}{\log^2 x}+\frac{238}{\log^3 x} + \frac{97}{\log^4 x} + \frac{30}{\log^5 x} + \frac{8}{\log^6 x}.$$
Now, subtracting $(\ref{epi})$ from $(\ref{pipi})$ we have

$$\pi^2(x) - \frac{ex}{\log x} \pi \Big( \frac{x}{e} \Big) < \frac{x^2}{\log^6 x} \Big(-1 + \frac{\epsilon_{M_a} (x) - \epsilon_{m_a} (x)}{\log x} \Big).$$
The right hand side is negative if

$$\log x > \epsilon_{M_a} (x_{a}) - \epsilon_{m_a} (x_{a}),$$
and so we can then choose $x_a' $ as some value which satisfies this.

\end{proof}

The aim is to reduce $\max(ex_{a},x_{a}' )$ so as to get the sharpest bound available using this method and modern estimates involving the prime counting function.  The next two subsections deal with deriving the explicit bounds on $\pi(x)$ that are required to invoke Lemma \ref{lem1}. 

\subsection{An estimate for Chebyshev's function}

We define Chebyshev's $\theta$-function for some $x \in \mathbb{R}$ to be

$$\theta(x) = \sum_{p \leq x} \log p$$
where the sum is over prime numbers. We now call on Theorem 1 of Trudgian \cite{trudgian}, which explicitly bounds the error in approximating $\theta(x)$ with $x$.

\begin{lem}
Let

$$\epsilon_0 (x) = \sqrt{\frac{8}{17 \pi}} X^{1/2} e^{-X}, \hspace{0.2in} X= \sqrt{(\log x)/R}, \hspace{0.2in}  R = 6.455.$$
Then

$$|\theta(x) - x | \leq x \epsilon_0(x), \hspace{0.2in} x \geq 149$$
\end{lem}

This is another form of the prime number theorem, though explicit and able to give us the estimates required to use Lemma \ref{lem1}. For any choice of $a>0$, it is possible to use the above lemma to find some $x_a >0$ such that

\begin{equation} \label{chebyshevbound}
| \theta(x) - x | < a \frac{x}{\log^5 x}
\end{equation}
for all $x>x_a$; we simply need to find the range of $x$ for which 

$$\sqrt{\frac{8}{17 \pi}} \bigg( \frac{\log x}{R} \bigg)^{1/4} e^{-\sqrt{(\log x)/R}} < a \frac{x}{\log^5 x}. $$
As this may yield large values of $x_a$, we write $x = e^y$ (also $x_a = e^{y_a}$) and take logarithms to get the equivalent inequality

\begin{equation} \label{solve}
\log c + \frac{21}{4} \log y \leq \sqrt{\frac{y}{R}}.
\end{equation}

In the next part, we will see how bounds of the form in ($\ref{chebyshevbound}$) can be manipulated to give the estimates on $\pi(x)$ required to use Lemma \ref{lem1}.

\subsection{Upper and lower bounds for $\pi(x)$}

Suppose that, for any $a>0$ and some corresponding $x_a>0$ we have

$$\theta(x) < x + a \frac{x}{\log^5 x}$$
for all $x > x_a$. The technique of partial summation gives us that

\begin{eqnarray*}
\pi(x) & < & \frac{x}{\log x} + \int_2^x \frac{dt}{\log^2 t} + a \frac{x}{\log^6 x} + a \int_2^x \frac{dt}{\log^7 t} \\ \\
& < & x \Big( \sum_{k=0}^{4} \frac{k!}{\log^{k+1} x} \Big) + (120+a) \frac{x}{\log^6 x} + (720 +a) \int_2^x \frac{dt}{\log^7 t}.
\end{eqnarray*}
We can estimate the remaining integral here by

\begin{eqnarray*}
\int_2^x \frac{dt}{\log^7 t} & < & \frac{x}{\log^7 x} + 7 \int_2^x \frac{dt}{\log^8 t} \\ \\
& < & \frac{x}{\log^7 x} + 7 \bigg( \int_2^{\sqrt{x}} \frac{dt}{\log^8 t} +  \int_{\sqrt{x}}^{x} \frac{dt}{\log^8 t}\bigg) \\ \\
& < &  \frac{x}{\log^7 x} + 7 \Big( \frac{\sqrt{x}}{\log^8 2} + \frac{2^8 x}{\log^8 x} \Big).
\end{eqnarray*}
Putting it all together we have that

$$\pi(x) < x \Big( \sum_{k=0}^{4} \frac{k!}{\log^{k+1} x} \Big) + M_a \frac{x}{\log^6 x}$$
for all $x>x_a$, where 

\begin{equation} \label{M}
M_a = 120 + a +\frac{a+720}{\log x_a} + \frac{1792 a + 1290240}{\log^2 x_a} + \Big( \frac{5040+7a}{\log^8 2} \Big) \frac{\log^6 x_a}{\sqrt{x_a}}.
\end{equation}

In an almost identical way, we can obtain for $x>x_a$ that

$$\pi(x) > x \Big( \sum_{k=0}^{4} \frac{k!}{\log^{k+1} x} \Big) + m_a \frac{x}{\log^6 x}$$
where

\begin{equation} \label{m}
m_a = 120 -a - \frac{a}{\log x_a} - \frac{1792}{\log^2 x_a} - 2 A \frac{\log^6 x_a}{x_a} - \frac{7 a \log^6 x_a}{\log^8 2 \sqrt{x_a}}
\end{equation}
and

$$A = \sum_{k=1}^{5} \frac{k!}{\log^{k+1} 2} \approx 1266.08.$$

\subsection{Numerical estimates}

Our method is as follows. We choose some $a>0$ such that we wish for

$$|\theta(x) - x| < a \frac{x}{\log^5 x} $$
to hold for $x>x_a = e^{y_a}$. We simply plug our desired value of $a$ into (\ref{solve}) and use \textsc{Mathematica} to search for some value of $y_a$, such that the inequality holds for all $x > e^{y_a}$.  We then use (\ref{M}) and (\ref{m}) to calculate two values $m_a$ and $M_a$ such that

$$ x \sum_{k=0}^{4} \frac{k!}{\log^{k+1}x}+ \frac{m_a x}{\log^6 x} < \pi(x) < x \sum_{k=0}^{4} \frac{k!}{\log^{k+1}x}+\frac{M_a x}{\log^6 x}$$
holds for $x>e^{y_a}$. Then by Lemma \ref{lem1}, we find some value $x_{a}' = e^{y_a'}$ (dependent on $a$, $m_a$ and $M_a$, and thus really only on $a$) such that Ramanujan's inequality is true for $x > \max(ex_a, x_a')$. 

One finds that small values of $a$, give rise to large values of $x_a$, yet small values of $x_a'$. Similiarly, large values of $a$ will yield small $x_a$ yet large values of $x_a '$. Of course, we want $x_a$ and $x_a'$ to be comparable, so that we might lower their maximum as much as possible. Thus, the idea is to select $a$ so that $ex_a$ and $x_a'$ are as close as possible.

It can be verified that choosing $a=3223$ gives $x_a = \exp(9656.8)$ with the values

$$m_a = -3103.33, \hspace{0.2in} M_a = 3343.48.$$
One then computes, using Lemma \ref{lem1} that $x_a' = \exp(9657.8)$ will work. This gives us Theorem \ref{thm:uncon}.

\section{Estimates on the Riemann hypothesis}

We now assume the Riemann Hypothesis and can therefore rely on Schoenfeld's conditional bound for the prime counting function:
\begin{thm}[Schoenfeld]\label{thm:schoen}
For $x\geq 2657$ we have
\begin{equation*} \label{schoen}
|\pi(x) - \text{li}(x) | < \frac{1}{8 \pi} \sqrt{x} \log x.
\end{equation*}
\end{thm}
\begin{proof}
See \cite{schoenfeld}.
\end{proof}

We now aim to improve on Theorem \ref{thm:hassani} of Hassani to the extent that a numerical computation to check the remaining cases become feasible. We have
\begin{lem}\label{lem:lowlim}
Assuming the Riemann Hypothesis, we have
$$\pi^2(x)<\frac{ex}{\log x}\pi\left(\frac{x}{e}\right)$$
for all $x\geq 1.15\cdot 10^{16}$.
\end{lem}
\begin{proof}
Platt and Trudgian \cite{platttrudgian} have recently confirmed that $\pi(x)<\text{li}(x)$ holds for $x \leq 1.2 \cdot10^{17}$. Together with Theorem \ref{thm:schoen} we see that

$$f(x)=\pi^2(x) - \frac{ex}{\log x} \pi \Big( \frac{x}{e} \Big)$$
is bounded above by

$$g(x) =  \text{li}^2(x) - \frac{ex}{\log x} \bigg( \text{li}\Big(\frac{x}{e} \Big) - \frac{1}{8 \pi} \sqrt{\frac{x}{e}} (\log x - 1) \bigg) $$
for all $x \geq 2657 e$. Berndt \cite[Ch.24]{Berndt} uses some elementary calculus to show that a similar function to the above is monotonically increasing over some range. One can use that same technique here to show that $g(x)$ is monotonically decreasing for all $x \geq 10^{16}$. Then, \textsc{Mathematica} can be used to show that

$$g(1.15 \cdot 10^{16}) \approx -3.211 \cdot 10^{19} <0$$
and thus $g(x)$ is negative for all $x \geq 1.15 \cdot 10^{16}$ and the lemma follows.
\end{proof}

\subsection{Computation}

It was stated in the introduction that the largest integer counterexample of (\ref{inequality}) up to $x=10^{11}$ occurs at $x = 38, 358, 837, 682$. In this subsection, we wish to show by computation that there are no counterexamples in the interval $[10^{11}, 1.15 \cdot 10^{16}]$.

As before, we write

$$f(x) = \pi^2(x) - \frac{ex}{\log x} \pi\bigg( \frac{x}{e} \bigg).$$
Note that $f$ is strictly decreasing between primes, so we could simply check that $f(p)<0$ for all primes $p$ in the required range. However, there are roughly $3.2\cdot 10^{14}$ primes to consider\footnote{Or precisely $319,870,505,122,591$.} and this many evaluations of $f$ would be computationally too expensive. Instead we employ a simple stepping argument.
\begin{lem}\label{lem:step}
Let $x_0$ be in the interval $[10^{11},1.15 \cdot 10^{16}]$ with $f(x_0) < 0$. Set $\epsilon=\sqrt{\pi^2(x_0)-f(x_0)}-\pi(x_0)$. Then $f(x) < 0$ for all $x\in[x_0, x_0 + \epsilon]$.
\end{lem}
\begin{proof}
We have for $x_0\geq e$ and $\epsilon>0$
\begin{eqnarray*}
f(x_0+\epsilon) & = & \pi^2(x_0+\epsilon)-\frac{e(x_0+\epsilon)}{\log(x_0+\epsilon)}\pi\left(\frac{x_0+\epsilon}{e}\right)\\
& \leq & \pi^2(x_0)+2\pi(x_0)\epsilon+\epsilon^2-\frac{ex_0}{\log x_0}\pi\left(\frac{x_0}{e}\right)\\
& = & f(x_0)+2\pi(x_0)\epsilon+\epsilon^2.
\end{eqnarray*}
Setting $f(x_0+\epsilon)=0$ and solving the resulting quadratic in $\epsilon$ gives us our lemma.
\end{proof}

Suppose we have access to a table of values of $\pi(x_i)$ with $x_{i+1}>x_i$ for all $i$. Then we can compute an interval containing $\pi(x)$ simply by looking up $\pi(x_i)$ and $\pi(x_{i+1})$ where $x_i\leq x \leq x_{i+1}$. Repeating this for $x/e$ we can determine an interval $[a,b]$ for $f(x)$. Assuming $b$ is negative, we can use Lemma \ref{lem:step} to step to a new $x$ and repeat.

Oliviera e Silva has produced extensive tables of $\pi(x)$ \cite{Oliviera2012}. Unfortunately, these are not of a sufficiently fine granularity to support the algorithm outlined above. In other words the estimates on $\pi(x)$ and $\pi(x/e)$ we can derive from these tables alone are too imprecise and do not determine the sign of $f(x)$ uniquely. We looked at the possibility of refining the coarse intervals provided by these tables using Montgomery and Vaughan's explicit version of the Brun-Titchmarsh \cite{MV} theorem but to no avail. Instead, we re-sieved the range $[1\cdot 10^{10},1.15\cdot 10^{16}]$ to produce exact values for $\pi(x_i)$ where the $x_i$ were more closely spaced. Table \ref{tab:pixs} provides the details.\footnote{At the same time we double checked Oliviera e Silva's computations and, as expected, we found no discrepancies.}

\begin{table}[ht] 
\caption{The Prime Sieving Parameters} 
\label{tab:pixs} 
\centering 
\begin{tabular}{c c c} 
\hline\hline 
From & To & Spacing\\[0.5 ex] \hline
$10^{10}$ & $10^{11}$ & $10^3$\\  
$10^{11}$ & $10^{12}$ & $10^4$\\  
$10^{12}$ & $10^{13}$ & $10^5$\\  
$10^{13}$ & $10^{14}$ & $10^6$\\  
$10^{14}$ & $10^{15}$ & $10^7$\\  
$10^{15}$ & $10^{16}$ & $10^8$\\  
$10^{16}$ & $1.15\cdot 10^{16}$ & $10^9$\\  
\hline\hline 
\end{tabular} 
\end{table} 

We used Kim Walisch's ``primesieve'' package \cite{Walisch2012} to perform the sieving and it required a total of about $300$ hours running on nodes of the University of Bristol's Bluecrystal cluster \cite{ACRC2014}.\footnote{Each node comprises two $8$ core Intel\textsuperscript{\circledR} Xeon\textsuperscript{\circledR} E5-2670 CPUs running at $2.6$ GHz and we ran with one thread per core.}

Using the stepping algorithm outlined above running against these tables, actually confirming that $f(x)<0$ for all $x\in[1\cdot 10^{11},1.15\cdot 10^{16}]$ took less than $5$ minutes on a single core. We had to step (and therefore compute $f(x)$) about $5.3\cdot 10^8$ times to span this range. We sampled at every $100,000$th step and Figure \ref{fig:plot} shows a log/log plot of $x$ against $-f(x)$ for these samples.\footnote{Actually we use the midpoint of the interval computed for $-f(x)$.}

\begin{figure}[tbp]
\centering
\fbox{\includegraphics[width=0.65\linewidth,angle=270]{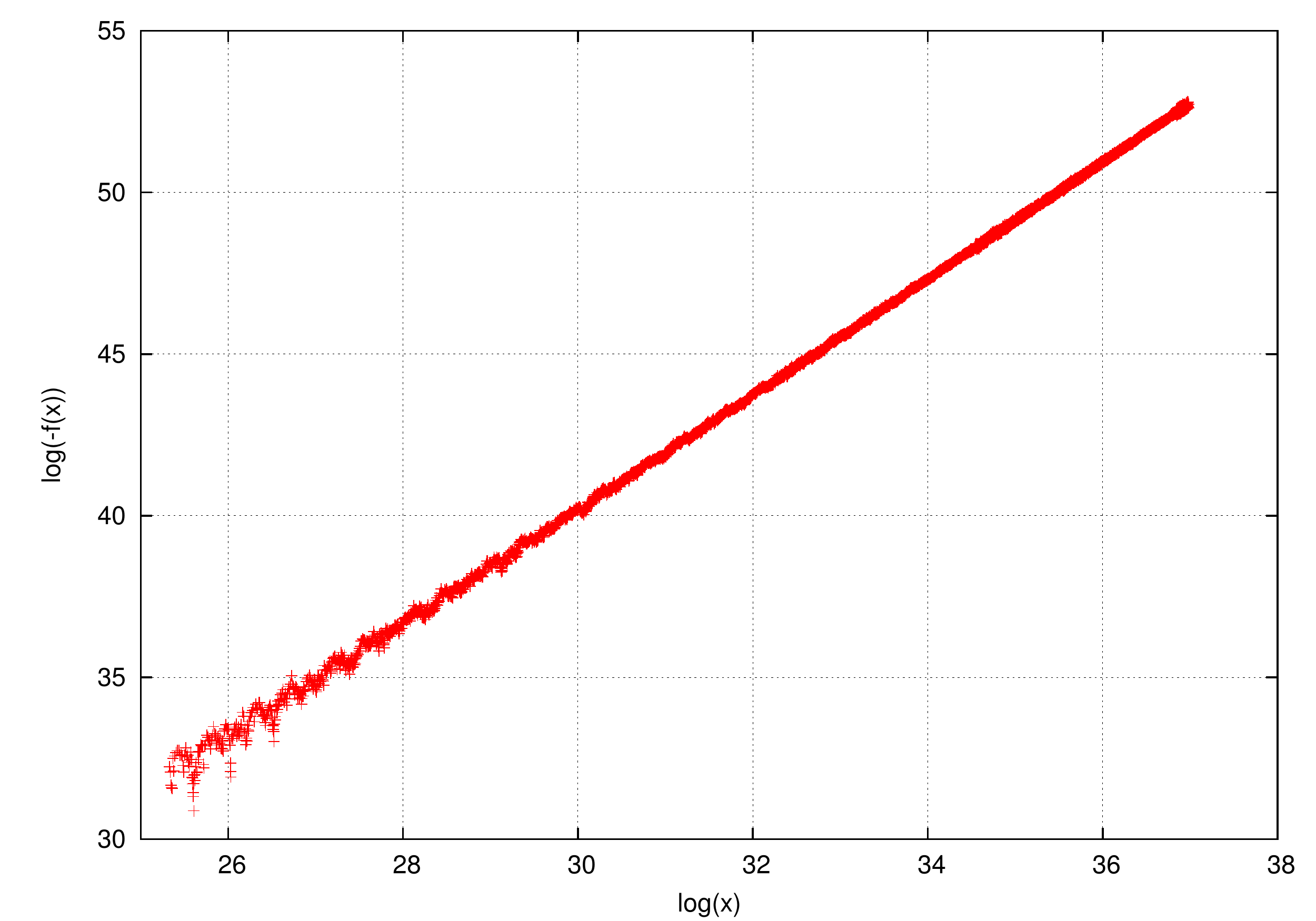}}
\caption{$\log(x)$ vs. $\log(-f(x))$.}
\label{fig:plot}
\end{figure}

No counter-examples to (\ref{inequality}) where uncovered by this computation and so we can now state
\begin{lem}\label{lem:comp}
For $x\in[10^{11},1.15\cdot 10^{16}]$ we have
$$\pi^2(x)<\frac{ex}{\log x}\pi\left(\frac{x}{e}\right).$$
\end{lem}

Lemmas \ref{lem:lowlim} and \ref{lem:comp} now give us Theorem \ref{thm:con}.

\clearpage

\bibliographystyle{plain}

\bibliography{bibliov3}

\begin{thebibliography}{1}

\bibitem{ACRC2014}
ACRC.
\newblock {BlueCrystal Phase 3 User Guide}, 2014.
\newblock \url{http://www.acrc.bris.ac.uk/acrc/}.

\bibitem{Berndt}
B.~C. Berndt.
\newblock {\em Ramanujan's Notebooks: Part IV}.
\newblock Springer-Verlag, 1993.

\bibitem{Hassani}
Mehdi Hassani.
\newblock {On an inequality of Ramanujan concerning the prime counting
  function}.
\newblock {\em The Ramanujan Journal}, 28(3):435--442, 2012.

\bibitem{MV}
H.~L. Montgomery and R.~C. Vaughan.
\newblock The large sieve.
\newblock {\em Mathematika}, 20:119--134, 12 1973.

\bibitem{Oliviera2012}
T.~Oliveira~e Silva.
\newblock {Tables of Values of $\pi(x)$ and $\pi^2(x)$}, 2012.
\newblock \url{http://www.ieeta.pt/~tos/primes.html}.

\bibitem{platttrudgian}
D.~J. Platt and T.~S. {Trudgian}.
\newblock {On the first sign change of $\theta(x)-x$}.
\newblock {\em arXiv:}, 2014.

\bibitem{schoenfeld}
L~Schoenfeld.
\newblock Sharper bounds for the {C}hebyshev functions {$\theta (x)$} and
  {$\psi (x)$}. {II}.
\newblock {\em Math. Comp.}, 30(134):337--360, 1976.

\bibitem{trudgian}
T.~S. {Trudgian}.
\newblock {Updating the error term in the prime number theorem}.
\newblock {\em arXiv:1401.2689}, 2014.

\bibitem{Walisch2012}
K.~Walisch.
\newblock {primesieve}.
\newblock \url{http://code.google.com/p/primesieve/}.

\end{thebibliography}

\end{document}